\documentclass[12pt,a4paper]{article}
\usepackage{amsmath,amssymb,amsthm,amstext}
\usepackage{cite}
\usepackage{indentfirst}
\usepackage{color}
\newtheorem{theorem}{Theorem}[section]
\newtheorem{lemma}[theorem]{Lemma}

\theoremstyle{definition}
\newtheorem{definition}[theorem]{Definition}
\theoremstyle{remark}
\newtheorem{remark}[theorem]{Remark}

\numberwithin{equation}{section}

\newcommand{\C}{\mathbb{C}}
\newcommand{\N}{\mathbb{N}}
\newcommand{\R}{\mathbb{R}}

\let\le\leqslant
\let\leq\leqslant

\let\geq\geqslant

\allowdisplaybreaks

\begin{document}

\title{Global mild solutions for the nonautonomous 2D Navier-Stokes equations with impulse effects}

\author{E. M. Bonotto\thanks{Instituto
de Ci\^encias Matem\'aticas e de Computa\c{c}\~ao, Universidade de
S\~ao Paulo, Campus de S\~ao Carlos, Caixa Postal 668,
S\~ao Carlos, SP, Brazil. E-mail: {\tt ebonotto@icmc.usp.br}.
Supported partially by CNPq grant 307317/2013-7 and FAPESP grant 2012/16709-6.},  \, J. G. Mesquita\thanks{Faculdade de Filosofia, Ci\^encias e Letras de Ribeir\~ao Preto, Departamento de Computa\c{c}\~{a}o e Matem\'{a}tica, Universidade de S\~{a}o Paulo, CEP 14.040-901, Ribeir\~ao Preto, SP, Brazil. E-mail: {\tt jgmesquita@ffclrp.usp.br}. Supported by FAPESP grant 2012/08473-2.} \, and \, R. P. Silva\thanks{Instituto de Geoci\^encias e Ci\^encias Exatas, Univ. Estadual Paulista J\'ulio de Mesquita Filho, Campus de Rio Claro, CEP 13.506-900, Rio Claro, SP, Brazil. Email: {\tt rpsilva@rc.unesp.br}. Supported partially by CNPq \#440371/2014-7 and FAPESP \#2014/16165-1.}}

\date{}

\maketitle

\begin{abstract}
The present paper deals with existence and uniqueness of global mild solutions for the 2D Navier-Stokes equations with impulses. Using the framework of nonautonomous dynamical systems, we extend previous results considering the 2D Navier-Stokes equations with impulse effects and allowing that the nonlinear terms are explicitly time-dependent. Additionally, we present sufficient conditions to obtain dissipativity (boundedness) for solutions starting in bounded sets.
\end{abstract}

\section{Introduction}

The Navier-Stokes equations (NSEs) represent a formulation of the Newton's laws of motion for a continuous distribution of matter in a fluid state, characterized by an inability to support shear stresses, see \cite{Doering-Gibbon}. The NSEs allow to determine the velocity field and the pressure of fluids confined in regions of the space, and they are used to describe many different physics phenomena as weather, water flow in tubes, ocean currents and others. Moreover, these equations are useful in several fields of knowledge such as petroleum industry, plasma physics, meteorology, thermo-hydraulics, among others (see \cite{RT} for instance). Due to this fact, these equations have been attracted to the attention of several mathematicians since they play an important role for applications. See \cite{boldrini, Alexandre, Doering-Gibbon, GRR, GRR1, GRR2, Jiu-Wang-Xin, rosa, RT, Teman} and the references therein. 

On the other hand, the theory of impulsive dynamical systems has been shown to be a powerful tool to model real-world problems in physics, technology, biology, among others. Because of this fact, the interest in the study of impulsive dynamical systems has increasing considerably. For recents trends on this subject we indicate the works \cite{BonottoDemuner, bonotto1, BBCC, Cortes, Davis, Feroe, Yang, Zhao} and the references therein.

However, the study of Navier-Stokes equations with impulse effects is really scarce. Motivated by this fact, in this paper, we investigate existence and uniqueness of mild solutions for the impulsive NSEs
\begin{equation}\label{Eq5}
\displaystyle\left\{\begin{array}{ll}
                      \displaystyle\frac{\partial u}{\partial t} + q(t)(u \cdot \nabla)u - \nu\Delta u +\nabla p = \phi(t,u), & (t,x) \in \left((0, +\infty)\setminus \displaystyle\bigcup_{k=1}^{+\infty}\{t_k\}\right) \times \Omega, \vspace{1mm}\\
                      {\rm div}\, u = 0, & (t,x) \in (0, +\infty) \times \Omega, \vspace{1mm}\\
                      u = 0, & (t,x) \in (0, +\infty) \times \partial\Omega, \vspace{1mm}\\
                      u(0, \cdot)= u_0 & x \in \Omega, \vspace{1mm}\\
                      u(t_k^+, \cdot) - u(t_k^-, \cdot) = I_k (u(t_k, \cdot)), & x\in\Omega, \; k=1, 2,\ldots ,
                    \end{array}
\right.
\end{equation}
where $\Omega$ is a bounded smooth domain in $\R^2$. Here $u = (u_1,u_2)$ denotes the velocity field of a fluid filling $\Omega$, $p$ is its scalar pressure and $\nu > 0$ is its viscosity. We will assume that $q$ is a bounded function, $\phi$ is a nonlinearity which will be specified later,  $\{t_k\}_{k \in \mathbb{N}} \subset (0, +\infty)$ is a sequence of impulse times such that $\displaystyle\lim_{t \rightarrow +\infty} t_k = +\infty$, 
$u(t_k, \cdot) = u(t_k^+, \cdot) = \displaystyle\lim_{\delta \rightarrow 0+}u(t_k + \delta, \cdot)$,
$u(t_k^-, \cdot) = \displaystyle\lim_{\delta \rightarrow 0+}u(t_k - \delta, \cdot)$
and
$I_k$, $k \in \N$, are the impulse operators.

Besides to impulsive actions in the system \eqref{Eq5}, we also allow that the external force $\phi$ is not continuous and depends on the solution $u$.  

We point out that the Navier-Stokes equations with impulses make sense physically and allow to describe more precisely the phenomena modeled by these equations, since $u$ represents the velocity of the field of a fluid and moreover, the external force $\phi$ in this case does not need to be continuous. It is well known that the phenomena which occur in the environment have impulsive behavior and the functions which model them have several discontinuities. Therefore, with this impulsive model, we intend to give a more precisely description of the Navier-Stokes equations.

The system \eqref{Eq5} without impulse conditions was studied in the classical monograph \cite{Cheban}, where $\phi$ is a function of time $t\in\mathbb{R}$. More precisely, the author studies existence and uniqueness of global mild solutions for the non-impulsive equation
\[
 \displaystyle\frac{\partial u}{\partial t} + q(t)(u \cdot \nabla)u - \nu\Delta u +\nabla p = \phi(t),
\]
subject to the conditions $\textrm{div}\, u = 0$ and $u|_{\partial \Omega} = 0$, where $\Omega$ is a bounded smooth domain in $\R^2$.

% and $(u \cdot \nabla) = \displaystyle \sum_{j=1}^2 u_j \frac{\partial }{\partial x_j} $.

%In this paper we consider NSE's in the presence of impulse effects. 

%To the best of our knowledge, the Navier-Stokes equations were investigated only to the non-impulsive case. Motivated by this fact,

%Let $H$ be a real or complex Hilbert space, $(\mathcal{M},d)$ be a metric space and $(\mathcal{M}, \sigma)$ be a dynamical system on $\mathcal{M}$. Denote by $\mathcal{L}^2 (H)$ the space of all continuous bilinear operators $B : H \times H \to H$  and by $C(A, B)$ the space of all continuous functions defined on $A$ taking values in $B$.

%In \cite{Cheban}, the author studies the existence and uniqueness of mild solutions for the 2D Navier-Stokes equations
%\[
% \displaystyle\frac{\partial u}{\partial t} + q(t)(u \cdot \nabla)u - \nu\Delta u +\nabla p = \phi(t),
%\]
%subject to the conditions $\textrm{div}\, u = 0$ and $u|_{\partial \Omega} = 0$, where $\Omega$ is an open bounded set with boundary $\partial \Omega \in C^2$ considering a general equation 
%\[
%u' + Au + B( \sigma(t, \omega))(u, u) = f( \sigma(t, \omega)),
%\]
%where $B \in C(\mathcal{M}, \mathcal{L}^2(H))$, $f \in C(\mathcal{M}, H)$ and $A$ is a (unbounded) linear operator. Note that $f$ and $B$ are non-stationary.

Our goal here is to write a weaker formulation of the system \eqref{Eq5} and then, we intend to investigate the existence and uniqueness of mild solutions. In order to do this, we start by considering some notations which can be found in \cite{rosa} and \cite{RT}, for instance.
 Let $\mathbb{L}^2 (\Omega) = (L^2 (\Omega))^2$ and $\mathbb{H}_0^1 (\Omega) = (H_0^1 (\Omega))^2$ endowed, respectively, with the inner products
$$(u,v) = \displaystyle \sum_{j=1}^2  \int_{\Omega} u_j \cdot v_j \ dx, \ \ \  u = (u_1, u_2),  \ v = (v_1, v_2) \in \mathbb L^2 (\Omega),$$
and
$$((u,v)) = \displaystyle \sum_{j=1}^2 \int_{\Omega}  \nabla u_j \cdot \nabla v_j dx, \ \ \  u = (u_1, u_2),  \ v = (v_1, v_2) \in \mathbb{H}^1_0 (\Omega)$$
and norms $| \cdot | = ( \cdot, \cdot)^{1/2}$ and $\| \cdot \| = (( \cdot, \cdot))^{1/2}$. 

%We notice that the norm $\| \cdot \|$ is equivalent to the usual one in $\mathbb H_0^1 (\Omega)$. 

Now, we consider the following sets:
$$\mathcal{E} = \{ v \in (C_0^{\infty}(\Omega))^2: \; \nabla \cdot v = 0  \ \textrm{in} \ \Omega\},$$
$$ V = \textrm{closure of} \ \mathcal{E} \ \textrm{in} \ \mathbb{H}_0^1 (\Omega)$$
and
$$ H = \textrm{closure of} \ \mathcal{E} \ \textrm{in} \ \mathbb{L}^2 (\Omega).$$

The space $H$ is a Hilbert space with the scalar product $(\cdot, \cdot)$ induced by $\mathbb{L}^2 (\Omega)$ and the space $V$ is a Hilbert space with the scalar product $((u, v))$ induced by $\mathbb{H}_0^1 (\Omega)$.

The space $V$ is contained in $H$, it is dense in $H$ and by the Poincare's Inequality, the inclusion $i:V  \hookrightarrow H$ is continuous. Denote by $V'$ and $H'$ the dual spaces of $V$ and $H$, respectively. The adjoint operator $i^*$ is linear and continuous from $H'$ to $V'$, $i^*(H')$ is dense in $V'$ and $i^*$ is one to one since $i(V) = V$ is dense in $H$. Moreover, by the Riesz representation Theorem, we can identify $H$ and $H'$ and write
$$V \subset H \equiv H' \subset V',$$
where each space is dense in the following one and the injections are continuous. %For details, see \cite{RT}.

As a consequence of the previous identifications, the scalar product in $H$, $(f,u)$, of $f \in H$ and $u \in V$ is the same as the duality product between $V'$ and $V$, $\langle f, u \rangle$, i.e.,
$$ \langle f, u \rangle = (f, u), \ \ \ \text{for all} \;  f \in H \;  \text{and} \; \text{for all} \; u \in V.$$

Also, for each $u \in V$, the form
$$v \in V \mapsto \nu((u, v)) \in \mathbb R$$
is linear and continuous on $V$. Therefore, there exists an element of $V'$ which we denote by $Au$ such that
$$ \langle Au, v \rangle = \nu((u, v)), \; \text{for all} \; v \in V.$$

Notice that the mapping $u \mapsto Au$ is linear, continuous and it is an isomorphism from $V$ to $V'$. %See, for instance, \cite{RT}.

%Based on it, let us now to consider the following weak formulation of \eqref{Eq5}, that is, to find
%$$u \in L^{\infty} (0, T; F) \cap L^2 ([0, T], E), \ \ \forall T > 0,$$
%such that

Based on it, we consider the following weak formulation of \eqref{Eq5}:
\begin{equation}\label{weak-Eq}
\left\{
\begin{array}{lll}
\displaystyle\frac{d}{dt} (u, v) + \nu ((u, v)) + b(t)(u,u,v) = \langle \phi(t,u), v \rangle, \ \ v \in V, \; t > 0, \; t \neq t_k, \vspace{1mm}\\
u(t_k) - u(t_k^-) = I_k(u(t_k^-)),  \ \  k \in \mathbb N, \vspace{1mm}\\
u(0) = u_0 \in H,
\end{array}
\right.
\end{equation}
where $\phi(t,u) \in V'$ and
$b(t): V \times V \times V \to \mathbb R$ is given by
$$b(t) (u, v, w) = q(t) \displaystyle\sum_{i, j = 1}^2 \displaystyle\int_{\Omega} u_i \displaystyle\frac{\partial v_j}{\partial x_i} w_j dx.$$
%$\langle \cdot, \cdot \rangle$ is the duality product between $V'$ and $V$ when we identify $H$ with its dual and 
%we assume that $\phi(t,u) \in V'$.
 The weak formulation \eqref{weak-Eq} is equivalent to the impulsive system
\begin{equation}\label{IS}
\left\{
\begin{array}{lll}
u'+ Au + B(t)(u, u) = \phi, \ \ \ \textrm{in} \ \  V', \ \ t > 0, \; t \neq t_k,\vspace{1mm}\\
u(t_k) - u(t_k^-) = I_k(u(t_k^-)), \ \ k \in \mathbb{N}, \vspace{1mm}\\
u(0) = u_0 \in H,
\end{array}
\right.
\end{equation}
where $u'= du/dt$, $A: V \to V'$ is the Stokes operator defined by
$$\langle Au, v \rangle = \nu((u,v)), \; \text{for all} \;  u, v \in V,$$
and $B(t): V \times V \to V'$ is a bilinear operator defined by
$$\langle B(t)(u, v), w \rangle = b(t) (u, v, w), \; \text{for all} \; u, v, w \in V.$$

In Section 2, we consider the following general impulsive system
\begin{equation}\label{IntroNS1}
\left\{
\begin{array}{lll}
u' + Au + B(\sigma(\cdot,\omega))(u,u) = f(\cdot, \sigma(\cdot,\omega), u),
  \quad t > 0, \; t \in I, \;  t \neq t_k, \; k\in\mathbb{N}, \vspace{1mm}\\
u(t_k) - u(t_k^-) = I_k (u(t_k^-)), \ \  k \in \mathbb{N},\vspace{1mm}\\
u(0) = u_0 \in H,
\end{array}
\right.
\end{equation}
where $f: I\times\mathcal{M}\times H\rightarrow H$ is a piecewise continuous function with respect to $t\in\mathbb{R}$, non-stationary and also depends on the solution $u$. All the conditions of system \eqref{IntroNS1} will be specified later. We prove the existence and uniqueness of global mild solutions for the system \eqref{IntroNS1} when $\mathcal{M}$ is compact, see Theorems \ref{T1.1} and \ref{T1.2}. The case when $\mathcal{M}$ is not compact is considered in Theorem \ref{T1.3}.

In Section 3, we prove 
 existence and uniqueness of global mild solutions for the 2D NSEs with impulses \eqref{Eq5} via system \eqref{IntroNS1}. 
We also give sufficient conditions to obtain dissipativity
for the system \eqref{Eq5}. All the results from this paper hold for the non-impulsive case.

%The present paper is divided as follows: the second and three sections are devoted to present the basic concepts and results concerning the impulsive dynamic systems. In the fourth section, we prove a result concerning the existence and uniqueness of a mild solution of Navier-Stokes equation with impulses and obtain some results concerning its properties. Finally, in the last section, we present some results about the existence of attractors of non-autonomous impulsive dynamic system associated to the Navier-Stokes equation with impulses and present some examples to illustrate them.

\section{Preliminaries}

Let $(\mathcal{M},d)$ be a compact metric space and $(\mathcal{M}, \mathbb R, \sigma)$ be a dynamical system on $\mathcal{M}$, i.e., $\sigma:\R \times \mathcal{M} \to \mathcal{M}$ is a continuous mapping which satisfies the following properties:
\begin{itemize}
\item[$i)$] $\sigma(0, \omega)= \omega$, \quad $\omega \in \mathcal{M}$;

\item[$ii)$] $\sigma(s,\sigma(t,\omega))=\sigma(s+t,\omega)$, \quad $t,s \in \R$, $\omega \in \mathcal{M}$.
\end{itemize}

Let $H$ be a real or complex Hilbert space and $\mathcal{L}(H) = \{T: H \rightarrow H: \, T \; \text{is linear }$ $\text{and bounded}\}$ equipped with the operational norm. Let $A:D(A) \subset H \to H$ be a self-adjoint operator such that
\begin{equation}\label{Coer}
{\rm Re}\, \langle Au, u \rangle \geq a |u|^2_H,
\end{equation}
for all $u\in D(A)$ and $a > 0$. It follows by \cite[Lemma 6.20]{Alexandre} that $\mathbb{C}\setminus (-\infty, a]\subset \rho(A)$ (resolvent of $A$), and there exists a constant $M\geq 1$ such that 
$$
\|(\lambda - A)^{-1}\|_{\mathcal{L}(H)} \le \frac{M}{|\lambda -a |}, \quad \lambda \in \Sigma_{a,\varphi}, \, \lambda \neq a,
$$
where $\Sigma_{a,\varphi}=\{\lambda \in \C: |{\rm arg}\, (\lambda-a)| \le \varphi\}$, $\varphi < \pi$. In particular, $A$ is sectorial. It follows that $-A$ generates an analytic semigroup $\{e^{-At}: t\geq 0\} \subset \mathcal{L}(H)$ which satisfies
\begin{equation}\label{EQ0}
\|e^{-At}\|_{\mathcal{L}(H)} \le K e^{-at},
\end{equation}
for some constant $K > 0$, where $a > 0$ comes from \eqref{Coer}.

%
%\vspace{2cm}
%
%a sectorial operator in $H$, i.e., $A$ is a closed densely defined linear operator such that for some $a \in \R$ and $\varphi \in (0, \frac{\pi}{2})$ the sector $\Sigma_{a,\varphi}=\{\lambda \in \C: \varphi \le |{\rm arg}\, (\lambda-a)| \le \pi\}$ is in the resolvent $\rho(A)$ of $A$ and there exists $M \geq 1$ such that
%$$
%\|(\lambda - A)^{-1}\|_{\mathcal{L}(H)} \le \frac{M}{|\lambda -a |}, \quad \lambda \in \Sigma_{a,\varphi}, \, \lambda \neq a.
%$$
%
%Therefore, $-A$ generates an analytic semigroup $\{e^{-At}: t\geq 0\} \subset \mathcal{L}(H)$ which satisfies
%\begin{equation}\label{EQ0}
%\|e^{-At}\|_{\mathcal{L}(H)} \le K e^{-at},
%\end{equation}
%%
%for some constant $K > 0$.

Assuming that $0 \in \rho(A)$, we consider the scale of Hilbert spaces $X^\alpha=D(A^\alpha)$ of fractional power of the operator $A$ endowed with the norm $\|\cdot \|_{X^\alpha}= \|A^\alpha \cdot \|_H$   ($X^0=H$). If $ \beta > \alpha \geq 0$, it is well known that $X^\beta$ is a dense subspace of $X^\alpha$ with continuous inclusion and 
\begin{equation} \label{eq:fractpow}
\|e^{-At}\|_{\mathcal{L}(X^\alpha, X^\beta)} \le C_{\alpha,\beta} \, t^{\alpha-\beta} e^{-at}.
\end{equation}

Consider $F$  a Hilbert space such that $H \subset F$ with inclusion dense and continuous. 
Denote by $\mathcal{L}(F, H)=\{T: F \rightarrow H: \, T \; \text{is linear and bounded}\}$ equipped with the operational norm.
We will assume that the semigroup $\{e^{-At}: t\geq 0\}$ satisfies:
\begin{enumerate}
\item[$i)$] $e^{-At} \in \mathcal{L}(F, H)$, \quad for all $t >0$;

\item[$ii)$] There exists $0 \le \alpha_1 <1$, such that
\begin{equation}\label{eq:sg-est-F}
\|e^{-At} \|_{\mathcal{L}(F, H)} \le K_1 t^{-\alpha_1} e^{-at}, \quad \text{for all}\; t >0.
\end{equation}

\end{enumerate}

%We assume the following condition
%\[
%Ae^{At} = e^{At}A,
%\]
%in the sense of $\mathcal{L}(F, E) = \{A: F \rightarrow E: \, A \; \text{is linear and bounded}\}$ equipped with the operational norm.
%

%Consider $V$ a subspace of $H$ with inclusion dense and continuous. 
%satisfying
%\begin{equation}\label{EQ3}
%\| e^{-At} \|_{\mathcal{L} (,E)} \leq K t^{-\alpha_1} e^{-at}, \ \ 0 \leq \alpha_1 < 1,
%\end{equation}
%

We denote by $\mathcal{L}^2 (H, F)$ the space of all continuous bilinear operators $\mathcal{B} : H \times H \to F$ equipped with the norm
\[
\|\mathcal{B}\|_{\mathcal{L}^2 (H, F)} = \sup \{ |\mathcal{B}(u, v)|_F : |u|_H \leq 1, \ |v|_H \leq 1 \}.
\]

Let $C(\mathcal{M}, \mathcal{L}^2(H, F))$ be the space of all continuous mapping $B: \mathcal{M} \to \mathcal{L}^2 (H, F)$ endowed with the norm
\[
\|B\|_{\infty} = \sup_{\omega \in \mathcal{M}}  \|B(\omega)\|_{\mathcal{L}^2 (H, F)}.
\]

The space $\Big( C(\mathcal{M}, \mathcal{L}^2(H, F)), \|\cdot\|_{\infty} \Big)$ is a Banach space, see \cite{Cheban}.

For all $u, v \in H$ and $\omega \in \mathcal{M}$, we have that
\begin{equation}\label{EQ5}
|B(\omega)(u, u) - B(\omega)(v, v)|_F \leq \|B\|_{\infty}(|u|_H + |v|_H ) |u - v|_H,
\end{equation}
%By the last inequality, it follows that on every ball $B[0, R] := \{ u \in E : |u|_E \leq R\}$, we have
%$$|B(w)(u_1, u_1) - B(w)(u_2, u_2)|_F \leq 2 C_B R |u_1 - u_2|_E$$
%for every $u_1, u_2 \in E$.
and also
\begin{equation}\label{EQ51}
|B(\omega)(u,u)|_F \leq \|B\|_{\infty}|u|_H^2.
\end{equation}

Let $\{t_k\}_{k \in \mathbb{N}}$ be a strictly increasing sequence in $(0, +\infty)$ such that $\displaystyle\lim_{t \rightarrow +\infty} t_k = +\infty$.
Let $I \subset \mathbb{R}$ be an interval and $f: I\times \mathcal{M}\times H \rightarrow H$ and $I_k: H \rightarrow H$, $k \in \mathbb{N}$,
be functions satisfying the following conditions:

\begin{enumerate}
\item[(C1)] For each fixed $t \in I$, $f(t, \cdot, \cdot)$ is continuous on $\mathcal{M}\times H$.

\item[(C2)] Let $\omega \in \mathcal{M}$ and $u \in H$. Then $\displaystyle\lim_{s \rightarrow t} f(s, \omega, u) = f(t, \omega, u)$ for all $t \neq t_k$, $k \in \mathbb{N}$,
the limit $\displaystyle\lim_{s \rightarrow t_k-} f(s, \omega, u)$ exists and  $\displaystyle\lim_{s \rightarrow t_k+} f(s, \omega, u) = f(t_k, \omega,u)$, for all $k \in \mathbb{N}$.

%\item[(C3)] There exists $\alpha>0$ such that
%\[
%{\rm Re}\, \langle Au, u \rangle \geq \alpha |u|^2_E,
%\]
%for all $u \in E$.
%\item[{\rm (ii)}] For every $u, v, \omega \in E$ and $w \in \Omega$, we have
%\begin{equation}\label{11.19}
%Re \langle B(\omega)(u,v), w \rangle = - Re\langle B(\omega)(u,w), v \rangle.
%\end{equation}

\item[(C3)] There is a bounded function $M : \mathbb R \to \mathbb R_+$, such that for 
any interval $[a,b] \subset I$, we have
$$\displaystyle \int_a^b |\phi(s)| | f(s, \omega, u)|_H ds \leq \displaystyle\int_a^b M(s) |\phi(s)| ds$$
for all $\phi \in L^1 [a,b]$, $\omega \in \mathcal{M}$ and $u \in H$.

\item[(C4)] There is a bounded function $L: \mathbb{R} \rightarrow \mathbb{R}_+$, such that for any interval $[a, b] \subset I$, we have
%\[
%\sup_{t\in [a, b], \omega \in \Omega}|f(t, \omega)|_X  < \infty.
%\]
\[
\int_a^b |\phi(s)| |f(s, \omega_1, u_1) - f(s, \omega_2, u_2)|_H \, ds \leq \int_a^b L(s)|\phi(s)|(d(\omega_1, \omega_2) + |u_1 - u_2|_H)ds
\]
for all $\phi \in L^1[a, b]$, $\omega_1, \omega_2 \in \mathcal{M}$ and $u_1, u_2 \in H$.

\item[(C5)]  There exists a constant $K_2>0$ such that
\[
\sup_{k\in\mathbb{N}}\sup_{u\in H}|I_k(u)|_H \le K_2.
\]

\item[(C6)] There exists a constant $K_{3}>0$ such that
\[
|I_k(u) - I_k(v)|_H \le K_3 |u- v|_{H}
\]
for all $u,v \in H$ and for all $k\in\mathbb{N}$.
\end{enumerate}

Now, given $\omega \in \mathcal{M}$ and assuming all the conditions above, we consider the following impulsive system in the state space $H$: 
\begin{equation}\label{NS1}
\left\{
\begin{array}{lll}
u' + Au + B(\sigma(\cdot,\omega))(u,u) = f(\cdot, \sigma(\cdot,\omega), u),
  \quad t > 0, \; t \in I, \;  t \neq t_k, \; k\in\mathbb{N}, \vspace{1mm}\\
u(t_k) - u(t_k^-) = I_k (u(t_k^-)), \ \  k \in \mathbb{N},\vspace{1mm}\\
u(0) = u_0 \in H.
\end{array}
\right.
\end{equation}

%Note that $u(t_k) = u(t_k^+)$, for all $k\in\mathbb{N}$.

\begin{remark} Since $\displaystyle\lim_{k \rightarrow +\infty} t_k = +\infty$, it is clear that given a closed interval $[0, T]$, there exists at most a finite number of moments of impulses $t_1, t_2,\ldots, t_n \in [0, T]$ such that $0 < t_1 < t_2 < \ldots < t_n \leq T$. Thus, given $T > 0$ there is an integer $n_T > 0$ such that $t_{n_T} \leq T < t_{n_T+1}$.
\end{remark}

Given $T>0$, we consider the space $PC^+ ([0,T], H)= \{u:[0, T]\to H : \ u \text{ is continuous at }\linebreak
 t \neq t_k, \text{ right-continuous at } t=t_k \text{ and the limit } \displaystyle \lim_{t \to t_k^-} u(t) \;\text{exists for all} \, k = 1, \ldots, n_T \} $. It is well known that the space $PC^+ ([0,T], H)$ endowed with the norm $\displaystyle \|u\|_{PC^+}= \sup_{t \in [0,T]} |u(t)|_{H}$ is a Banach space.

In the sequel, we present the definition of a mild solution for the system \eqref{NS1}.

\begin{definition}\label{mild-solution} Let $[0, T]\subset I$. We say that $u \in PC^+ ([0,T], H)$ is a \emph{mild solution} of \eqref{NS1} if $u$ satisfies the following integral equation:
\begin{equation}\label{NS2}
u(t) = \displaystyle\left\{
\begin{array}{lcc}
 e^{-At} u_{0} + \displaystyle \int_{0}^t e^{-A(t-s)}g(s, \omega, u(s)) ds,  &  \text{if} &  0 \leq t < t_1, \\
e^{-A(t-t_1)}[u(t_1^-) + I_1(u(t_1^-))] +\displaystyle \int_{t_1}^t e^{-A(t-s)}g(s, \omega, u(s)) ds,   &  \text{if} & t_1 \leq t < t_2,   \\
e^{-A(t-t_2)}[u(t_2^-) + I_2(u(t_2^-))] +\displaystyle \int_{t_2}^t e^{-A(t-s)}g(s, \omega, u(s)) ds,   &  \text{if} &   t_2 \leq t < t_3, \\
& \vdots & \\
e^{-A(t-t_{k})}[u(t_{k}^-) + I_k(u(t_{k}^-))] +\displaystyle  \int_{t_{k}}^t e^{-A(t-s)}g(s, \omega, u(s)) ds,  &\text{if}  & t_{k} \leq t \leq T,
\end{array}
\right.
\end{equation}
 where $0 < t_1 < \ldots < t_k \leq T < t_{k+1}$ are the impulse times ($k=n_T$) and $g(s, \omega, u(s)) = -B({\sigma} (s,\omega))(u(s),u(s)) + f(s,{\sigma} (s, \omega), u(s))$, $s \in [0, T]$.
System \eqref{NS2} can be rewritten in the following way
\[
u(t) = e^{-At} u_{0} + \displaystyle\int_{0}^{t} e^{-A(t-s)}g(s, \omega, u(s))ds +
 \displaystyle\sum_{0 < t_i <t} e^{-A(t-t_i)}I_i(u(t_i^-)).
\]
\end{definition}

Given $K \subset H$, we consider the following space of functions:
\begin{align*}
P&C^{+}_1([0,T]\times {K} \times \mathcal{M}, H) =\{ \varphi:[0,T]\times {K}\times \mathcal{M} \to H: \, \text{ for all } (u,\omega) \in {K} \times \mathcal{M}, \\
&  \varphi(\cdot,u,\omega) \in PC^+( [0, T], H) \text{ and for all } t \in [0,T], \,  \varphi(t,\cdot,\cdot): K \times \mathcal{M} \to H \text{ is continuous} \}.
\end{align*}

Theorem \ref{T1.1} ensures that the nonautonomous system \eqref{NS1} admits a unique mild solution in the sense of Definition \ref{mild-solution}.

\begin{theorem}\label{T1.1} Let $u_0 \in H$ and $r >0$. Assume that \eqref{EQ0}, \eqref{eq:sg-est-F}, \eqref{EQ5} and conditions  (C1) - (C6) hold. Then there exist positive numbers $\delta = \delta(u_0, r) > 0$, $T = T(u_0, r) > 0$ and a function $\varphi: [0, T]\times \overline{B(u_0, \delta)}\times \mathcal{M} \rightarrow H$
satisfying the following conditions:
\begin{enumerate}
\item[$i)$] $\varphi (0, u_0, \omega) = u_0$, for all $\omega \in \mathcal{M}$;

\item[$ii)$] $| \varphi (t, u, \omega) - u_0|_H \leq r$ for all $(t,u,\omega) \in [0, T]\times \overline{B(u_0, \delta)} \times \mathcal{M}$; 

\item[$iii)$] $\varphi \in PC^{+}_1([0,T] \times \overline{B(u_0, \delta)} \times \mathcal{M}, \overline{B(u_0, r)})$.
%For all $(u,\omega) \in \overline{B(u_0, \delta)} \times \mathcal{M}$, $\varphi(\cdot,u,\omega) \in PC^+( [0, T], B[u_0, r])$ and for all $t \in [0,T]$, $\varphi(t,\cdot,\cdot): \overline{B(u_0, \delta)} \times \mathcal{M} \to E$ is continuous.
\end{enumerate}
Moreover, the function $u:[0,T] \to H$ defined by $u(t)=\varphi(t,u_0,\omega)$ is the unique mild solution of system \eqref{NS1}.
\end{theorem}
\begin{proof}  Let $\delta > 0$ and $T > 0$ be such that $[0, T]\subset I$. Given $\varphi \in PC^{+}_1 ([0,T]\times \overline{B(u_0, \delta)}\times \mathcal{M}, H)$, we define
\[
S\varphi(t, u, \omega) = e^{-At}u + \int_{0}^t e^{-A(t-s)} g(s,\omega,\varphi(s)) ds + \displaystyle\sum_{0 < t_i <t} e^{-A(t-t_i)}I_i(\varphi(t_i^-)),
\]
where $\varphi(s) = \varphi(s, u, \omega)$ and $g(s, \omega, \varphi(s)) = -B(\sigma(s,\omega))(\varphi(s),\varphi(s)) + f(s, \sigma(s, \omega), \varphi(s))$, for all $s\in [0, T]$, $u \in  \overline{B(u_0, \delta)}$ and $\omega \in \mathcal{M}$. 
 Since functions in $PC^{+}_1([0, T]\times \overline{B(u_0, \delta)} \times \mathcal{M},\overline{B(u_0, r)})$ are bounded, we can consider the distance
\[
d_{\infty}(\varphi_1, \varphi_2) = \sup\{|\varphi_1(t, u, \omega) - \varphi_2(t, u, \omega)|_H: 0 \leq t \leq T, \, u \in \overline{B(u_0, \delta)}, \, \omega \in \mathcal{M}\},
\]
for $\varphi_1, \varphi_2 \in PC^{+}_1([0, T]\times \overline{B(u_0, \delta)} \times \mathcal{M},\overline{B(u_0, r)})$. It is not difficult to see that $(PC^{+}_1([0, T]\times \overline{B(u_0, \delta)} \times \mathcal{M},\overline{B(u_0, r)}), d_{\infty})$ is a complete metric space. For convenience, let us denote $\Gamma(\delta, T, r) = PC^{+}_1([0, T]\times \overline{B(u_0, \delta)} \times \mathcal{M},\overline{B(u_0, r)})$ and $\Gamma(\delta, T) = PC^{+}_1([0, T]\times \overline{B(u_0, \delta)} \times \mathcal{M}, H)$.

\vspace{0.3cm}

\textbf{Assertion 1:}  $S \in  C(\Gamma(\delta, T, r), \Gamma(\delta, T))$. 
\vspace{0.3cm}

In fact, at first note that $S\varphi \in \Gamma(\delta, T)$ for all $\varphi \in \Gamma(\delta, T, r)$.

Now, let $\varphi_1, \varphi_2 \in \Gamma(\delta, T, r)$ and $(t, u, \omega) \in [0, T]\times \overline{B(u_0, \delta)} \times \mathcal{M}$. By Condition (C4) there is a bounded function $L: \mathbb{R} \rightarrow \mathbb{R}_+$ such that
\[
\int_0^t e^{-a(t-s)} |f(s,\sigma(s, \omega), \varphi_1(s)) - f(s, \sigma(s, \omega), \varphi_2(s))|_H \,ds \leq
\]
\begin{equation}\label{EQ2.9}
\leq \int_0^t L(s)|\varphi_1(s) - \varphi_2(s)|_H \,ds \leq N Td_{\infty}(\varphi_1, \varphi_2),
\end{equation}
where $N = \displaystyle\sup_{s \in [0, T]}|L(s)|$.

Then, using \eqref{EQ0}, \eqref{eq:sg-est-F}, \eqref{EQ5}, \eqref{EQ2.9} and Condition (C6),
we have
\[
|S\varphi_1(t, u, \omega) - S\varphi_2(t, u, \omega)|_H  \leq
\]
\[
\leq \displaystyle\int_{0}^{t} \left|e^{-A(t - s)} \left[B(\sigma(s, \omega))(\varphi_1(s), \varphi_1(s)) - B(\sigma(s, \omega))(\varphi_2(s), \varphi_2(s)) \right]\right|_H ds +
\]
\[
+ \int_{0}^{t} \left|e^{-A(t -s)}\left[f(s, \sigma(s, \omega), \varphi_1(s)) - f(s, \sigma(s, \omega), \varphi_2(s))\right]\right|_H ds+  
\]
\[
 +  \displaystyle\sum_{0 < t_i <t} |e^{-A(t-t_i)}[I_i(\varphi_1(t_i^-)) - I_i(\varphi_2(t_i^-))]|_H \leq
\]
\[
\leq 2\|B\|_{\infty}K_1(r + |u_0|_H) d_{\infty}(\varphi_1, \varphi_2)\int_0^t(t-s)^{-\alpha_1}e^{-a(t-s)}ds +
\]
\[
+KNT %\dfrac{T^{-\beta_1 + 1}}{-\beta_1 +1}
d_{\infty}(\varphi_1, \varphi_2) +  KK_3d_{\infty}(\varphi_1, \varphi_2)\sum_{0 < t_i <t} e^{-a(t-t_i)} \leq
\]
\begin{equation}\label{EQNS1}
\leq  \left(2\|B\|_{\infty}K_1(r + |u_0|_H)\dfrac{T^{-\alpha_1 + 1}}{-\alpha_1 + 1} +  KNT
%\dfrac{T^{-\beta_1 + 1}}{1-\beta_1} 
+ KK_3n_{T}\right)d_{\infty}(\varphi_1, \varphi_2),
\end{equation}
where $n_{T}$ is the number of impulses on the interval $[0, T]$.
Hence,  $S \in  C(\Gamma(\delta, T, r), \Gamma(\delta, T))$. 

\vspace{0.3cm}

\textbf{Assertion 2:} There are $\delta_1 = \delta_1(u_0, r) \in (0, \delta)$ and $T_1 = T_1(u_0, r)\in(0, T)$ such that $S: \Gamma(\delta_1, T_1, r) \rightarrow \Gamma(\delta_1, T_1, r)$. \vspace{0.3cm}

In fact, let $\varphi \in \Gamma(\delta, T, r)$ and $(t, u, \omega) \in [0, T]\times \overline{B(u_0, \delta)}  \times \mathcal{M}$. %Using \eqref{EQ2.1}
By \eqref{EQ0} and Condition $(C4)$, one can obtain a bounded function $L: \mathbb{R}\rightarrow \mathbb{R}_+$ such that
\[
 \left| \int_{0}^t e^{-A(t-s)}\left[f(s, \sigma(s, \omega), \varphi(s)) - f(s, \sigma(s, \omega), 0)\right] ds \right|_H \leq
\]
\[
\int_0^t K
%(t-s)^{-\beta_1}
e^{-a(t-s)} |f(s, \sigma(s, \omega), \varphi(s)) - f(s, \sigma(s, \omega), 0)|_Hds \leq
\]
\begin{equation}\label{EQ2.11}
\leq K\int_0^t L(s)
%(t-s)^{-\beta_1}
|\varphi(s)|_H ds \leq K N 
%\dfrac{T^{-\beta_1 + 1}}{-\beta_1 +1}
(|u_0|_H + r)T,
\end{equation}
where $N = \displaystyle\sup_{s \in [0, T]}|L(s)|$. 

Let $m(\delta, T) = \sup\left\{|e^{-At}u - u_0|_H: \, t \in [0, T], \; u \in  \overline{B(u_0, \delta)}\right\}$
and $M = \displaystyle\sup_{s \in [0, T]} | M(s)|$, where $M$ is the function given by Condition (C3).
Then, using \eqref{EQ0}, \eqref{eq:sg-est-F}, \eqref{EQ51}, \eqref{EQ2.11}, Condition (C3) and Condition (C5), we obtain
$$
\left|S \varphi(t, u, \omega) - u_0 \right|_H \leq \left| e^{-At}u - u_0 \right|_{H} +  \left| \int_{0}^t e^{-A(t -s)} B(\sigma(s, \omega))(\varphi(s), \varphi(s)) ds \right|_H+ $$
$$ + \left| \int_{0}^t e^{-A(t-s)} f(s, \sigma(s, \omega), \varphi(s)) ds \right|_H + \left|\displaystyle\sum_{0 < t_i <t} e^{-A(t-t_i)}I_i(\varphi(t_i^-))\right|_H \leq$$
\[
 \leq  m(\delta, T) + \int_{0}^t K_1e^{-a(t-s)} (t - s)^{-\alpha_1} \|B\|_{\infty} | \varphi(s)|^2_H ds +
 \]
 \[
+ \left | \int_{0}^t e^{-A(t-s)}\left[f(s, \sigma(s, \omega), \varphi(s)) - f(s, \sigma(s, \omega), 0)\right] ds \right|_H +
  \left| \int_{0}^t e^{-A(t-s)}f(s, \sigma(s, \omega), 0)ds \right|_H +
\]
\[
 + \sum_{0 < t_i <t} Ke^{-a(t-t_i)}|I_i(\varphi(t_i^-))|_H \leq
\]
\[
 \leq m(\delta, T) + K_1\|B\|_{\infty} (|u_0|_H + r)^2 \frac{T^{-\alpha_1 + 1}}{1 - \alpha_1}+ KK_2n_T +
\]
\[
+KNT(|u_0|_H + r)
+  \int_{0}^t K e^{-a(t-s)} M(s)  ds \leq
\]
\[
 \leq m(\delta, T) + K_1\|B\|_{\infty} (|u_0|_H + r)^2 \frac{T^{-\alpha_1 + 1}}{1 - \alpha_1} + KN 
% \dfrac{T^{-\beta_1 + 1}}{-\beta_1 +1}
 (|u_0|_H + r)T +
 \]
 \[
+ KMT
% \dfrac{T^{-\beta_1 + 1}}{-\beta_1 + 1} 
 + KK_2n_T := d_1 (u_0, r, \delta, T).
 \]

Now, we note that $d_1 (u_0, r, \delta, T) \rightarrow 0$
as $\delta \to 0$ and $ T \to 0$.
Thus, there are $\delta_1 = \delta_1(u_0, r) > 0, \delta_1 < \delta$, and $T_1 = T_1(u_0, r) > 0, T_1 < T$, such that $d_1(u_0, r, \delta', T') \leq r$ for all $\delta' \in (0, \delta_1]$ and $T' \in (0, T_1]$.

\vspace{0.3cm}

\textbf{Assertion 3:} There exist $T_0 = T_0(u_0, r)>0$ and $\delta_0 = \delta_0(u_0, r)> 0$ such that \linebreak$S: \Gamma(\delta_0, T_0, r) \rightarrow \Gamma(\delta_0, T_0, r)$ is a contraction.

\vspace{0.3cm}

In fact, take $T_2 > 0$ such that
\[
2\|B\|_{\infty}K_1(r + |u_0|_H)\dfrac{T_2^{-\alpha_1 + 1}}{-\alpha_1 + 1} +  KNT_2 
%\dfrac{T_2^{-\beta_1 + 1}}{1-\beta_1} 
+ KK_3n_{T_2} < 1.
\]
It is enough to take $\delta_0 = \delta_1$ and $T_0 = \min\{T_1, T_2\}$ to conclude Assertion 3.

In conclusion, by the Banach fixed point Theorem, there exists a unique function $\varphi \in \Gamma(\delta_0, T_0, r)$ satisfying the system \eqref{NS2} on the interval $[0,T_0]$ and the result follows.
\end{proof}

Theorem \ref{T1.2} gives sufficient conditions for the mild solution of system \eqref{NS1} to be prolongated on $\mathbb{R}_+$.

\begin{theorem}\label{T1.2}  Suppose that $I = \mathbb{R}_+$ and the conditions of Theorem \ref{T1.1} hold. If the mild solution $\varphi(t, u_0, \omega)$ of system \eqref{NS1} is bounded, then it may be prolonged on $\mathbb{R}_+$.
\end{theorem}
\begin{proof} By Theorem \ref{T1.1}, $\varphi(t, u_0, \omega)$ is the unique solution of system  \eqref{NS1}
passing through the point $u_0\in H$ at time $t=0$. This solution is defined on some maximal interval $[0, \alpha_{(u_0, \omega)})$. Let $\varphi(t, u_0, \omega)$ be bounded and suppose that $\alpha_{(u_0, \omega)} < \infty$. If $\alpha_{(u_0, \omega)} \neq t_k$ for all $k\in \mathbb{N}$, then 
defining $\varphi(\alpha_{(u_0, \omega)}, u_0, \omega) = \displaystyle\lim_{t\rightarrow \alpha_{(u_0, \omega)}-}\varphi(t, u_0, \omega)$, it follows that $\varphi(t, u_0, \omega)$ may be extend on the interval $[0, \alpha_{(u_0, \omega)}]$ which is a contradiction. Now, suppose that $\alpha_{(u_0, \omega)} = t_k$ for some $k\in\mathbb{N}$. Since the limit $\displaystyle\lim_{t\rightarrow t_k-}\varphi(t, u_0, \omega) = \varphi(t_k^-, u_0, \omega)$ exists and $I_k(\varphi(t_k^-, u_0, \omega)) \in H$, then we may use the proof of Theorem \ref{T1.1} and extend $\varphi(t, u_0, \omega)$ in some interval $[t_k, t_k + \epsilon)$, $\epsilon > 0$, with $\varphi(t_k, u_0, \omega) = \varphi(t_k^-, u_0, \omega) + I_k(\varphi(t_k^-, u_0, \omega))$ 
 which is a contradiction. Hence, $\alpha_{(u_0, \omega)} = +\infty$.
\end{proof}

By following the proofs of Theorems \ref{T1.1} and \ref{T1.2}, we can state the next result which deals with existence and uniqueness of global mild solutions for the system  \eqref{NS1} when $\mathcal{M}$ is not necessarily compact.

\begin{theorem}\label{T1.3}  Let $u_0 \in H$ and $r >0$. Assume that \eqref{EQ0}, \eqref{eq:sg-est-F}, \eqref{EQ5} and conditions  (C1) - (C6) hold. Suppose that $\mathcal{M}$ is not necessarily compact and $\|B\|_{\infty} < \infty$.
Then there exist positive numbers $\delta = \delta(u_0, r) > 0$, $T = T(u_0, r) > 0$ and a function $\varphi: [0, T]\times \overline{B(u_0, \delta)}\times \mathcal{M} \rightarrow H$
satisfying the following conditions:
\begin{enumerate}
\item[$i)$] $\varphi (0, u_0, \omega) = u_0$ for all $\omega \in \mathcal{M}$;

\item[$ii)$] $| \varphi (t, u, \omega) - u_0|_H \leq r$ for all $(t,u,\omega) \in [0, T]\times \overline{B(u_0, \delta)} \times \mathcal{M}$; 

\item[$iii)$] $\varphi \in PC^{+}_1([0,T] \times \overline{B(u_0, \delta)} \times \mathcal{M}, \overline{B(u_0, r)})$.
\end{enumerate}
Moreover, the function $u:[0,T] \to H$ defined by $u(t)=\varphi(t,u_0,\omega)$ is the unique mild solution of system \eqref{NS1}. If $u(t)$ is bounded and $I=\mathbb{R}_+$, then it can be prolonged on $\mathbb{R}_+$.
\end{theorem}
\vspace{.2cm}

\begin{remark} Suppose that \eqref{EQ0}, \eqref{eq:sg-est-F}, \eqref{EQ5} and conditions  (C1) - (C4) hold. By the proofs of the previous results we obtain the existence and uniqueness of global mild solutions for the following non-impulsive system:
\[
\left\{
\begin{array}{lll}
u' + Au + B(\sigma(t,\omega))(u,u) = f(t, \sigma(t,\omega), u),
  \quad t > 0, \vspace{1mm}\\
u(0) = u_0 \in H.
\end{array}
\right.
\]
\end{remark}

\section{The 2D Navier-Stokes equations with impulses}

In this section, we present conditions to obtain the existence and uniqueness of global mild solutions for the following 2D Navier-Stokes equations with impulses
\begin{equation}\label{l1}
\displaystyle\left\{\begin{array}{ll}
                      \displaystyle\frac{\partial u}{\partial t} + q(t)(u \cdot \nabla)u - \nu\Delta u +\nabla p = \phi(t, u), & (t,x) \in \left((0, +\infty)\setminus \displaystyle\bigcup_{k=1}^{+\infty}\{t_k\}\right) \times \Omega, \vspace{1mm}\\
                      {\rm div}\, u = 0, & (t,x) \in (0, +\infty) \times \Omega, \vspace{1mm}\\
                      u = 0, & (t,x) \in (0, +\infty) \times \partial\Omega, \vspace{1mm}\\
                      u(0, \cdot)= u_0(\cdot) & x \in \Omega, \vspace{1mm}\\
                    u(t_k, \cdot) - u(t_k^-, \cdot) = I_k(u(t_k^-, \cdot)), &  x\in\Omega, \; k=1, 2,\ldots ,
                    \end{array}
\right.
\end{equation}
where $\Omega$ is an open and bounded set in $\mathbb R^2$ with $\partial \Omega \in C^2$,
$u = (u_1,u_2)$ is the velocity field of a fluid, $p$ is the scalar pressure, $\nu > 0$ is the kinematic viscosity of the fluid, $\phi=\phi(t, u) \in \mathbb R^2$ is the external body force, $q(t)$ is a bounded function,  $\{t_k\}_{k \in \mathbb{N}} \subset (0, +\infty)$ is a sequence of impulses such that $\displaystyle\lim_{t \rightarrow +\infty} t_k = +\infty$ and
$I_k$ is the impulse operator for each $k \in \mathbb N$.
%We denote by $u(t) \in \mathbb R^2$ and $p(t) \in \mathbb R$ respectively, the velocity and pressure of the fluid at the point in $\Omega$ and at time $t \geq 0$.

Let
$$\mathcal{E} = \{u \in (C_0^{\infty}(\Omega))^2: \; \nabla \cdot u = 0  \ \textrm{in} \ \Omega\},$$
$$ V = \textrm{closure of} \ \mathcal{E} \ \textrm{in} \ \mathbb{H}_0^1 (\Omega)$$
and
$$ H = \textrm{closure of} \ \mathcal{E} \ \textrm{in} \ \mathbb{L}^2 (\Omega),$$
where $\mathbb{L}^2 (\Omega) = (L^2 (\Omega))^2$ and $\mathbb{H}_0^1 (\Omega) = (H_0^1 (\Omega))^2$ are endowed, respectively, with the inner products
\[
(u,v) = \displaystyle \sum_{j=1}^2  \int_{\Omega} u_j \cdot v_j \ dx, \ \ \  u = (u_1, u_2),  \ v = (v_1, v_2) \in \mathbb L^2 (\Omega),
\]
%$$(u,v) = \displaystyle\int_{\Omega} u \cdot v \ dx, \ \ \  u, v \in \mathbb L^2 (\Omega),$$
and
$$((u,v)) = \displaystyle  \sum_{j=1}^2 \int_{\Omega}   \nabla u_j \cdot \nabla v_j dx, \ \ \  u = (u_1, u_2),  \ v = (v_1, v_2) \in \mathbb{H}^1_0 (\Omega),$$
and norms $| \cdot | = ( \cdot, \cdot)^{1/2}$ and $\| \cdot \| = (( \cdot, \cdot))^{1/2}$.

We assume the following general hypotheses throughout this section:

\begin{enumerate}
\item[(H1)] $\phi: \mathbb{R}_+\times \mathbb{R}^2 \rightarrow \mathbb{R}^2$ is a bounded function such that for each fixed $t \in \mathbb{R}_+$, $\phi(t, \cdot)$ is continuous on $\mathbb{R}^2$.

\item[(H2)] Let $x \in \mathbb{R}^2$. Then $\displaystyle\lim_{s \rightarrow t} \phi(s, x) = \phi(t, x)$ for all $t \neq t_k$, $k \in \mathbb{N}$,
the limit $\displaystyle\lim_{s \rightarrow t_k-} \phi(s, x)$ exists and  $\displaystyle\lim_{s \rightarrow t_k+} \phi(s, x) = \phi(t_k, x)$, for all $k \in \mathbb{N}$.

\item[(H3)] There is $C > 0$ such that $|\phi(s, x) - \phi(s, y)| \leq C|x-y|$ for all $s \in \mathbb{R}_+$ and for all $x, y \in \mathbb{R}^2$.

%\item[(H3)] There is a bounded function $M : \mathbb R \to \mathbb R_+$, such that for 
%any interval $[a,b] \subset I$, we have
%$$\displaystyle\int_a^b \beta(s) | \phi(s, u)|_H ds \leq \displaystyle\int_a^b M(s) \beta(s) ds$$
%for all $\beta \in L^1 [a,b]$ and $u \in H$.
%
%
%\item[(H4)] There is a bounded function $L: \mathbb{R} \rightarrow \mathbb{R}_+$, such that for any interval $[a, b] \subset I$, we have
%\[
%\int_a^b \beta(s) |\phi(s, u_1) - \phi(s, u_2)|_Hds \leq \int_a^b L(s)\beta(s)|u_1 - u_2|ds
%\]
%for all $\beta \in L^1[a, b]$ and $u_1, u_2 \in H$.

\item[(H4)]  There exists a constant $C_1>0$ such that
\[
\sup_{k\in\mathbb{N}}\sup_{x\in \mathbb{R}^2}|I_k(x)| \le C_1.
\]

\item[(H5)] There exists a constant $C_{2}>0$ such that
\[
|I_k(x) - I_k(y)| \le C_2 |x- y|
\]
for all $x,y \in \mathbb{R}^2$ and for all $k\in\mathbb{N}$.
\end{enumerate}

 Now, denote by $P$ the corresponding orthogonal projection $P: \mathbb{L}^2 (\Omega) \rightarrow H$ and set the operators
\[
A = -\nu P \Delta
\]
and
\[
\mathcal{B}(t)(u, v) = q(t)P( (u \cdot \nabla) v).
\]
It is well known that the Stokes operator $A$ is positive self-adjoint with domain $D(A)$ dense in $H$, $0 \in \rho(A)$ and $A^{-1}$ is compact. Also, there exists $\alpha>0$ such that
\begin{equation}\label{JER1}
\langle Au, u \rangle \geq \alpha |u|^2_H,
\end{equation}
for all $u \in H$. We also have the following orthogonality property of the nonlinear term which is fundamental and expresses the conservation of energy by the inertial forces:
\begin{equation}\label{JER2}
 \langle \mathcal{B}(t)(u,v), v \rangle = 0
\end{equation}
for all $u, v \in H$ and for all $t\in\mathbb{R}_+$. For the above properties see, for instance, \cite{Cheban}, \cite{Constantin} and \cite{Teman}.

%For every $u, v, w \in H$ and $\omega \in \mathcal{M}$, we have
%\begin{equation}\label{JER2}
%Re \langle B(\omega)(u,v), w \rangle = - Re\langle B(\omega)(u,w), v \rangle.
%\end{equation}

 We set the Hilbert spaces $X^\alpha$, $\alpha \in(0, 1]$, as the domain of the powers of $A$ and we have
\[
V = X^{\frac{1}{2}} \quad \text{and} \quad |u|_{V} = |\nabla u|.
\]

Applying $P$ in the equation
\[
\displaystyle\frac{\partial u}{\partial t} + q(t)(u \cdot \nabla)u - \nu\Delta u +\nabla p = \phi(t, u),
\]
we obtain the evolution equation
\begin{equation}\label{I2}
u'+ Au + \mathcal{B}(t)(u, u) = \mathcal{F}(t, u),
\end{equation}
where $\mathcal{F}(t, u) = P\phi(t, u)$ for all $t>0$ and $u\in H$,  $\langle Au, v \rangle = \nu((u, v))$ for all $u, v \in H$
and 
\[
\langle \mathcal{B}(t)(u, u), w \rangle = q(t) \displaystyle\sum_{i, j = 1}^2 \displaystyle\int_{\Omega} u_i \displaystyle\frac{\partial u_j}{\partial x_i} w_j dx \quad \text{for all} \quad u, w \in H.
\]

We also assume that:
\begin{enumerate}
\item[(A)] $\mathcal{F} \in PC^+(\mathbb{R}_+\times H, H)$;
\item[(B)] $\mathcal{B} \in C(\mathbb{R}_+, L^2(H, F))$, where $F = D(A^{-\delta})$ for some $0 < \delta < 1$.
\end{enumerate}

% Let $\mathcal{F} \in PC(\mathbb{R}\times V, H)$ $(X:= H)$ and $\mathcal{B} \in C(\mathbb{R}, L^2(H, D(A^{-\delta})))$ $(F:= D(A^{-\delta}))$. 

Denote $Y$ by $C(\mathbb{R}_+, L^2(H, F))\times PC^+(\mathbb{R}_+\times H, H)$ and let $(Y, \mathbb{R}_+, \sigma)$ be the semidynamical system of translations, that is, $\sigma(t, g) = g_t$ for all $g \in Y$ and $t\geq 0$. Now, set
 \[
 \mathcal{M} := \mathcal{H}(\mathcal{B}, \mathcal{F}) = \overline{\{(\mathcal{B}_{\tau}, \mathcal{F}_{\tau}): \, \tau \in \mathbb{R}_+\}},
 \]
where $\mathcal{B}_{\tau}(t) = \mathcal{B}(t + \tau)$ for all $t \in \mathbb{R}_+$ and $\mathcal{F}_{\tau}(t, u) = \mathcal{F}(t + \tau, u)$ for all $t \in \mathbb{R}_+$ and $u\in H$. If 
$(\widetilde{\mathcal{B}}, \widetilde{\mathcal{F}}) \in \mathcal{M}$ and $\tau\geq 0$ we consider
$\widetilde{\mathcal{B}}_{\tau}(t) = \widetilde{\mathcal{B}}(t+\tau)$ and
$\widetilde{\mathcal{F}}_{\tau}(t, u) = \widetilde{\mathcal{F}}(t+\tau, u)$ for all $(t, u) \in \mathbb{R}_+\times H$. 

According to \cite{Cheban}, the equation
\begin{equation}\label{L2}
u' + Au + \widetilde{\mathcal{B}}(t)(u,u) = \widetilde{\mathcal{F}}(t, u),
\end{equation}
where $(\widetilde{\mathcal{B}}, \widetilde{\mathcal{F}}) \in \mathcal{H}(\mathcal{B}, \mathcal{F})$, is called the 
$\mathcal{H}-$class along with the equation \eqref{I2}.

Define the mapping $B: \mathcal{M} \rightarrow L^2(H, F)$ by
\[
B(\omega) = B(\widetilde{\mathcal{B}}, \widetilde{\mathcal{F}}) := \widetilde{\mathcal{B}}(0)
\]
and the mapping $f: \mathbb{R}_+\times\mathcal{M}\times H \rightarrow H$ by
\[
f(t, \omega, u) = f(t, (\widetilde{\mathcal{B}}, \widetilde{\mathcal{F}}), u) := \widetilde{\mathcal{F}}(0, u).
\]
Then equation \eqref{L2} can be rewritten in the form
\begin{equation}\label{NSNA}
u' + Au + B(\sigma(t, \omega))(u,u) = f(t, \sigma(t, \omega), u).
\end{equation}

From \eqref{JER2}, we obtain 
\begin{equation}\label{JER3}
 \langle B(\omega)(u,v), v \rangle = 0
\end{equation}
for all $u, v \in H$ and for all $\omega \in \mathcal{M}$.

\begin{lemma}\label{LEMMA1} $\displaystyle\sup_{\omega \in \mathcal{M}}  \|B(\omega)\|_{\mathcal{L}^2 (H, F)} < \infty$.
\end{lemma}
\begin{proof} Since $q(t)$ is bounded, there is $L > 0$ such that $|q(t)| \leq L$ for all $t\geq 0$. Then
\[
 \sup_{\omega \in \mathcal{M}}  \|B(\omega)\|_{\mathcal{L}^2 (H, F)} =  
\sup_{\omega\in\{(\mathcal{B}_{\tau}, \mathcal{F}_{\tau}): \, \tau \in \mathbb{R}_+\}}\|B(\omega)\|_{\mathcal{L}^2 (H, F)} 
=\sup_{\tau\in \mathbb{R}_+}  \|\mathcal{B}(\tau)\|_{\mathcal{L}^2 (H, F)} =
\]
\[
=\sup_{\tau\in \mathbb{R}_+} \sup_{|u|_H\leq 1, \, |v|_H\leq 1}|\mathcal{B}(\tau)(u, v)|_F = 
\sup_{\tau\in \mathbb{R}_+} \sup_{|u|_H\leq 1, \, |v|_H\leq 1}|q(\tau)P( (u \cdot \nabla) v)|_F \leq
\]
\[
 \leq L\sup_{|u|_H\leq 1, \, |v|_H\leq 1} \|(u \cdot \nabla) v) \|_{H} \leq  L\sup_{|u|_H\leq 1, \, |v|_H\leq 1}\|u\|_H \|\nabla v \|_H \leq 
\]
\[
\leq L\sup_{|u|_H\leq 1, \, |v|_H\leq 1} \|u\|_H \| v \|_V = L\sup_{|u|_H\leq 1, \, |v|_V\leq 1} \|u\|_H \| v \|_V  \leq L.
\]
Therefore, we have the desired result.
\end{proof}

Since $\phi$ is bounded and we have Lemma \ref{LEMMA1}, we may consider in $\mathcal{M}$ the metric $d_{\mathcal{M}}$ given by
\[
d_{\mathcal{M}}(\omega_1, \omega_2) = d_{\mathcal{M}}((\widetilde{\mathcal{B}}_1, \widetilde{\mathcal{F}}_1), (\widetilde{\mathcal{B}}_2, \widetilde{\mathcal{F}}_2)) = \|\widetilde{\mathcal{B}}_1 - \widetilde{\mathcal{B}}_2\|_{\infty} + \|\widetilde{\mathcal{F}}_1 -\widetilde{\mathcal{F}}_2\|_{PC^+_{\infty}},
\]
where  $\displaystyle \|\widetilde{\mathcal{F}}\|_{PC^+_{\infty}}= \sup_{(t, u) \in [0,+\infty)\times H} |\widetilde{\mathcal{F}}(t, u)|_{H}$ and $\|\widetilde{\mathcal{B}}\|_{\infty} = \displaystyle\sup_{\omega \in \mathcal{M}}  \|\widetilde{\mathcal{B}}(\omega)\|_{\mathcal{L}^2 (H, F)}$.

\vspace{.3cm}

Let us consider the following impulsive system associated to \eqref{NSNA}:
\begin{equation}\label{WeakEq}
\left\{
\begin{array}{lll}
u' + Au + B(\sigma(t, \omega))(u,u) = f(t, \sigma(t, \omega), u), \ \ u \in H, \; t > 0, \; t \neq t_k, \vspace{1mm}\\
u(t_k) - u(t_k^-) = I_k(u(t_k^-)),  \ \  k \in \mathbb N, \vspace{1mm}\\
u(0) = u_0,
\end{array}
\right.
\end{equation}
which is a weak formulation of \eqref{l1}. In the next lines, we show that $f$ satisfies the conditions 
(C1), (C2), (C3) and (C4) presented in Section 2. This will help us to show that system \eqref{WeakEq} admits a unique global mild solution, see Theorems \ref{TeoNS} and \ref{TEO2.6} in the sequel.

\begin{lemma}\label{LNS1}  The mapping $f: \mathbb{R}_+\times\mathcal{M}\times H \rightarrow H$ satisfies the conditions (C1), (C2), (C3) and (C4).
\end{lemma}
\begin{proof} First, let us show that $f$ satisfies condition (C1). Let 
$t \in \mathbb{R}_+$ be fixed. Take $(\omega_n, u_n), (\omega_0, u_0) \in \mathcal{M}\times H$, $n=1,2,\ldots$, such that 
\[
d_{\mathcal{M}}(\omega_n, \omega_0)  \rightarrow 0 \quad \text{and} \quad |u_n - u_0|_H \rightarrow 0
\]
as $n \rightarrow +\infty$. Note that $\omega_0 = (\widetilde{\mathcal{B}}, \widetilde{\mathcal{F}})$ and $\omega_n = (\widetilde{\mathcal{B}}_{n}, \widetilde{\mathcal{F}}_{n})$, $n=1,2,3,\ldots$. Moreover, 
\[
\|\widetilde{\mathcal{F}}_n -\widetilde{\mathcal{F}}\|_{PC^+_{\infty}} \rightarrow 0 \quad 
\text{as} \quad n\rightarrow +\infty.
\]
Now, since $(\widetilde{\mathcal{B}}_{n}, \widetilde{\mathcal{F}}_{n}), (\widetilde{\mathcal{B}}, \widetilde{\mathcal{F}}) \in \mathcal{M}$, $n=1,2,3,\ldots$, then there are sequences $\{\tau_k^n\}_{k\in\mathbb{N}}$ and $\{s_k\}_{k\in\mathbb{N}}$ in $\mathbb{R}_+$ such that
\[
\widetilde{\mathcal{F}}_{n}(r, u) = \lim_{k\rightarrow +\infty}\mathcal{F}_{\tau_k^n}(r, u) \quad \text{and} \quad \widetilde{\mathcal{F}}(r, u) = \lim_{k\rightarrow +\infty}\mathcal{F}_{s_k}(r, u)
 \]
for each $(r, u) \in \mathbb{R}_+\times H$ and $n\in\mathbb{N}$. 
Then, 
\[
|f(t, \omega_n, u_n) - f(t, \omega_0, u_0)|_H = |\widetilde{\mathcal{F}}_n(0, u_n) - \widetilde{\mathcal{F}}(0, u_0)|_H = \lim_{k\rightarrow +\infty} |\mathcal{F}(\tau^n_ k, u_n) - \mathcal{F}(s_k, u_0)|_H
\]
\[
\leq \lim_{k\rightarrow +\infty} |\mathcal{F}(\tau^n_ k, u_n) - \mathcal{F}(\tau^n_ k, u_0)|_H
 + \lim_{k\rightarrow +\infty} |\mathcal{F}(\tau^n_ k, u_0) - \mathcal{F}(s_k, u_0)|_H
\]
\[
= \lim_{k\rightarrow +\infty} |P\phi(\tau^n_ k, u_n) - P\phi(\tau^n_ k, u_0)|_H
 + |\widetilde{\mathcal{F}}_{n}(0, u_0) - \widetilde{\mathcal{F}}(0, u_0)|_H
\]
\[
\leq   \|P\|C|u_n - u_0|_H + |\widetilde{\mathcal{F}}_{n}(0, u_0) - \widetilde{\mathcal{F}}(0, u_0)|_H,
\]
where the last inequality follows by Condition (H3). Hence, $|f(t, \omega_n, u_n) - f(t, \omega_0, u_0)|_H \rightarrow 0$ as $n\rightarrow +\infty$.

The Condition (H2) implies in Condition (C2).

In order to show that Condition (C3) holds, we define the function $M : \mathbb R \to \mathbb R_+$ by $M(t) = \sup\{|\phi(s, u)|_H: \, s \in\mathbb{R}_+, u \in H\}$ for all $t\in\mathbb{R}$, which is well defined since $\phi$ is bounded by Condition (H1).  Note that $M$ is constant. Given $\omega = (\widetilde{\mathcal{B}}, \widetilde{\mathcal{F}})\in \mathcal{M}$, there is a sequence $\{r_n\}_{n\in\mathbb{N}}$ in $\mathbb{R}_+$ such that
\[
f(t, \omega, u) = \widetilde{\mathcal{F}}(0, u) = \lim_{n\rightarrow +\infty}\mathcal{F}_{r_n}(0, u) = \lim_{n\rightarrow +\infty}P\phi(r_n, u),
\]
for all $t\geq 0$ and $u\in H$.
Then, for $[a, b] \subset \mathbb{R}_+$, we have
\[
\displaystyle\int_a^b |\psi(s)| | f(s, \omega, u)|_H ds = \int_a^b|\psi(s)||\widetilde{\mathcal{F}}(0, u)|_Hds  \leq\displaystyle\int_a^b M(s)|\psi(s)|ds,
\]
for all $\psi \in L^1 [a,b]$, $\omega \in \mathcal{M}$ and $u \in H$.

Finally, we need to verify the Condition (C4). Define $L: \mathbb{R} \rightarrow \mathbb{R}_+$ by $L(t) = C + 1$, $t\in\mathbb{R}$, where $C$ is the constant of the Condition (H3). Then, given $[a, b] \subset \mathbb{R}_+$, we have
\[
\int_a^b |\psi(s)| |f(s, \omega_1, u_1) - f(s, \omega_2, u_2)|_Hds = \int_a^b|\psi(s)||\widetilde{\mathcal{F}}_1(0, u_1) - \widetilde{\mathcal{F}}_2(0, u_2)|_Hds \leq
\]
\[
\leq \int_a^b|\psi(s)|\left(|\widetilde{\mathcal{F}}_1(0, u_1) - \widetilde{\mathcal{F}}_1(0, u_2)|_H + |\widetilde{\mathcal{F}}_1(0, u_2) - \widetilde{\mathcal{F}}_2(0, u_2)|_H\right)ds \leq
\]

\[
\leq \int_a^b |\psi(s)|\left(C\|P\||u_1-u_2|_H + \|\widetilde{\mathcal{F}}_1 -\widetilde{\mathcal{F}}_2\|_{PC^+_{\infty}}\right)ds \leq 
\]
\[
\leq
\int_a^b L(s)|\psi(s)| (d_{\mathcal{M}}(\omega_1, \omega_2) + |u_1 - u_2|_{H})ds,
\]
for all $\psi \in L^1[a, b]$, $\omega_1 = (\widetilde{\mathcal{F}}_1, \widetilde{\mathcal{B}}_1), \omega_2 = (\widetilde{\mathcal{F}}_2, \widetilde{\mathcal{B}}_2) \in \mathcal{M}$ and $u_1, u_2 \in H$.
\end{proof}

\vspace{.3cm}

\begin{lemma}\label{LNS222} $\sup\{|f(t, \omega, u)|_H: \, t\geq0, \omega \in\mathcal{M}, u\in H\}< \infty$.
\end{lemma}
\begin{proof} 
If $t\geq 0$, $u\in H$ and $\omega = (\widetilde{\mathcal{B}}, \widetilde{\mathcal{F}})\in\mathcal{M}$, we have
\[
|f(t, \omega, u)|_H = |\widetilde{\mathcal{F}}(0, u)|_H = \lim_{n\rightarrow +\infty}|\mathcal{F}_{\tau_n}(0, u)|_H = \lim_{n\rightarrow +\infty}|P\phi(\tau_n, u)|_H \leq \|P\|\eta \leq \eta,
\]
where $\eta>0$ is a bound of $\phi$ since it is bounded by Condition (H1). Hence, the result follows.
\end{proof}

By Lemma \ref{LNS222}, we may define $\|f\|_1 = \sup\{|f(t, \omega, u)|_H: \, t\geq0, \omega \in\mathcal{M}, u\in H\}$.%, for all $f: I\times\mathcal{M}\times H \rightarrow H$ satisfying conditions (C1), (C2), (C3) and (C4).

From Theorem \ref{T1.3}, we have the following straightforward result of existence and uniqueness of mild solutions.

\begin{theorem}\label{TeoNS} Under conditions (H1)-(H5), (A) and (B), 
 the system \eqref{WeakEq} admits a unique mild solution $u: [0, T]\rightarrow H$ defined in some interval $[0, T]$ satisfying $u(0)=u_0$.
\end{theorem}

The mild solution of system \eqref{WeakEq} may be prolonged on $\mathbb{R}_+$, see Theorem \ref{TEO2.6}.

\begin{theorem}\label{TEO2.6} Suppose that $\displaystyle\sum_{i=1}^{+\infty}|I_i(u)|_H = \Gamma < \infty$ for all $u \in H$. Then:
\begin{enumerate}
\item[$a)$] The solution $\varphi(t,  u_0, \omega)$ of the impulsive system  \eqref{WeakEq} is bounded and therefore, it may be prolonged on $\mathbb{R}_+$;
\item[$b)$] $|\varphi(t, u_0, \omega)|_H \leq 2C(|u_0|_H) + \Gamma$, for all $t \geq 0$, $\omega \in \mathcal{M}$ and $u_0 \in H$, where 
\[
C(r) = \displaystyle\left\{
\begin{array}{ccc}
r  & \text{if}  & r \geq  \dfrac{\|f\|_1}{\alpha}, \\ \\
 \dfrac{\|f\|_1}{\alpha} &  \text{if} &   r\leq \dfrac{\|f\|_1}{\alpha}
\end{array}  \right.
\]
and $\alpha$ is given by \eqref{JER1}.
\end{enumerate}
\end{theorem}
\begin{proof} $a)$ Let $u_0\in H$ and $\omega \in \mathcal{M}$. By Theorem \ref{TeoNS}, there exists a unique solution $\varphi(t, u_0, \omega)$ of equation \eqref{WeakEq} passing through $u_0$ at time $t=0$ and defined on some interval $[0, T_{(u_0, \omega)})$. Let $\psi(t; 0, u_0, \omega)$ be the solution of the non-impulsive equation \eqref{NSNA} such that $\psi(0) = u_0$. Note that
\[
\varphi(t, u_0, \omega) = \psi(t; 0, u_0, \omega) \quad \text{for} \quad t \in [0, T_{(u_0, \omega)}) \cap [0, t_1).
\]
Define
 $\eta(t) = |\psi(t; 0, u_0, \omega)|_H^2$ for $t \in [0, T_{(u_0, \omega)})\cap [0, t_1)$ and denote $\psi(t; 0, u_0, \omega)$ by $\psi(t)$. Using \eqref{JER1} and \eqref{JER3}, we obtain 
\[
\eta'(t)  =  2 \langle \psi'(t), \psi(t) \rangle =  - 2 \langle A \psi(t), \psi(t) \rangle - 2 \langle B(\sigma(t, w))(\psi(t), \psi(t)), \psi(t) \rangle +
\]
\[
+ \ 2 \langle f(t, \sigma(t, \omega), \psi(t)), \psi(t)\rangle \leq  -2\alpha|\psi(t)|_H^2 + 2\|f\|_1|\psi(t)|_H.
\]
 Thus,  
$$\eta' \leq - 2 \alpha \eta + 2 \| f \|_1\eta^{1/2},$$
which implies that $\eta(t) \leq v(t)$ for all $t \in [0, T_{(u_0, \omega)})\cap [0, t_1)$, where   
$v(t)$ is an upper solution of $v' = -2\alpha v + 2\|f\|_1v^{1/2}$ such that $v(0) = \eta(0) = |u_0|^2_H.$ Then
\[
\eta(t) \leq \left[ \left( |u|_H - \displaystyle\frac{\| f\|_1}{\alpha} \right) e^{-\alpha t} + \displaystyle\frac{\| f\|_1}{\alpha} \right]^2,
\]
that is,
\[
|\varphi(t, u_0, \omega)|_H = |\psi(t; 0, u_0, \omega)|_H \leq \left(|u_0|_H - \dfrac{\|f\|_1}{\alpha}\right)e^{-\alpha t} + \dfrac{\|f\|_1}{\alpha}, \quad  t \in [0, T_{(u_0, \omega)})\cap [0, t_1).
\]

Now, if $t \in [0, T_{(u_0, \omega)})\cap [t_1, t_2)$, we use the previous argument and we obtain
\begin{eqnarray*}
|\varphi(t, u_0, \omega)|_H & \leq & \left(|\psi(t_1; 0, u_0, \omega) + I_1(\psi(t_1; 0, u_0, \omega))|_H - \dfrac{\|f\|_1}{\alpha}\right)e^{-\alpha(t-t_1)} + \dfrac{\|f\|_1}{\alpha}\\
&\leq & \left(|u_0|_H - \dfrac{\|f\|_1}{\alpha}\right)e^{-\alpha t} + |I_1(\psi(t_1; 0, u_0, \omega))|_He^{-\alpha(t-t_1)} + \dfrac{\|f\|_1}{\alpha}.
\end{eqnarray*}

Continuing with this reasoning, if $t \in [0, T_{(u_0, \omega)})\cap [t_k, t_{k+1})$, we get
\begin{equation}\label{EQ3.5}
|\varphi(t, u_0, \omega)|_H \leq  \left(|u_0|_H - \dfrac{\|f\|_1}{\alpha}\right)e^{-\alpha t} + \sum_{i=1}^{k}|I_i(\psi(t_i; t_{i-1}, u_{i-1}^+, \omega))|_H e^{-\alpha(t-t_i)} + \dfrac{\|f\|_1}{\alpha},
\end{equation}
where we denote $t_0 = 0$, $u_0^+ = u_0$ and $u_i^+ = \psi(t_i; t_{i-1}, u_{i-1}^+, \omega) + I_i( \psi(t_i; t_{i-1}, u_{i-1}^+, \omega))$ for $i=1,2,\ldots$.

Since there exists a finite number of impulses on the interval $[0, T_{(u_0, \omega)})$
and $\displaystyle\sum_{i=1}^{+\infty}|I_i(u)|_H = \Gamma < \infty$ for all $u \in H$, then by \eqref{EQ3.5} we get that $\varphi(t, u_0, \omega)$ is bounded. Consequently, by Theorem \ref{T1.3}, it can be prolonged on $\mathbb{R}_+$. 

$b)$ Note that $|\varphi(t, u_0, \omega)|_H \leq |u_0|_H + \dfrac{\|f\|_1}{\alpha} + \Gamma$, for all $t\geq 0$, $u_0\in H$ and $\omega \in \mathcal{M}$. Therefore, the result holds.
\end{proof}

%For the non-impulsive system \eqref{NSNA}, we can state the following result.
%
%
%\begin{theorem}\label{TEO2.66} 
%\begin{enumerate}
%\item[$a)$] The solution $\psi(t,  u, \omega)$ of the non-impulsive system  \eqref{NSNA} is bounded and therefore, it may be prolonged on $\mathbb{R}_+$;
%\item[$b)$] $|\psi(t, u, \omega)| \leq 2C(|u|_H)$, for all $t \geq 0$, $\omega \in \mathcal{M}$ and $u \in H$, where 
%\[
%C(r) = \displaystyle\left\{
%\begin{array}{ccc}
%r  & \text{if}  & r \geq  \dfrac{\|f\|}{\alpha}, \\ \\
% \dfrac{\|f\|}{\alpha} &  \text{if} &   r\leq \dfrac{\|f\|}{\alpha}
%\end{array}  \right.
%\]
%and $\alpha$ is given by \eqref{JER1}.
%\end{enumerate}
%\end{theorem}

The system \eqref{WeakEq} is called \textit{bounded dissipative} if there is a nonempty bounded set $B_0 \subset H$ such that for each bounded set $B \subset H$ there exists $T = T(B) > 0$ such that $\varphi(t, u_0, \omega) \in B_0$ for all $t\geq T$, $u_0\in B$ and $\omega\in\mathcal{M}$. In this case, $B_0$ is called a \textit{bounded attractor} for the system \eqref{WeakEq}.
In the next result, we obtain dissipativity for the system \eqref{WeakEq}.

\begin{theorem}\label{TEO2.7}  Suppose that $\displaystyle\sum_{i=1}^{+\infty}|I_i(u)|_H = \Gamma < \infty$ for all $u \in H$. Then the system \eqref{WeakEq} is bounded dissipative.
\end{theorem}
\begin{proof} By the proof of Theorem \ref{TEO2.6}, we have 
\[
\displaystyle\lim_{t\rightarrow +\infty}\sup_{|u_0|_H \leq r, \; \omega \in \mathcal{M}}|\varphi(t, u_0, \omega)|_H \leq \dfrac{\|f\|_1}{\alpha} + \Gamma,
\]
 for all $r > 0$. Hence, the set $B_0 = \left\{u \in H: \, |u|_H \leq \dfrac{\|f\|_1}{\alpha} + \Gamma\right\}$ is a bounded attractor for the system \eqref{WeakEq}. 
\end{proof}

In the last result, we present an estimative between two solutions with different initial data
in the same fiber $\omega \in \mathcal{M}$.

\begin{theorem}\label{TEO2.8} Suppose that $\displaystyle\sum_{i=1}^{+\infty}|I_i(u)|_H = \Gamma < \infty$ for all $u \in H$. Let $r_0 = 2\dfrac{\|f\|_1}{\alpha} + \Gamma$.
Then, for each $k \in \mathbb{N}$, we have
\[
|\varphi(t, u_1, \omega) - \varphi(t, u_2, \omega)|_H \leq (1 + C_2)^ke^{-(\alpha - 2\|B\|_{\infty}r_0 - C)t}|u_1 - u_2|_H,
\]
for all $t \in [t_k, t_{k+1})$, $u_1, u_2 \in \overline{B(0, \alpha^{-1}\|f\|_1)}$ and $\omega \in \mathcal{M}$.
\end{theorem}
\begin{proof} Let $t \in \mathbb{R}_+$, $\omega \in \mathcal{M}$, $u_1, u_2 \in \overline{B(0, \alpha^{-1}\|f\|_1)}$ and define
$\eta(t) = \psi(t; 0, u_1, \omega) - \psi(t; 0, u_2, \omega)$, where  
$\psi_i := \psi(t; 0, u_i, \omega)$ is the solution of equation \eqref{NSNA} without impulses defined on $\mathbb{R}_+$ and passing through $u_i$ at time $t=0$,  $i=1,2$.  By Theorem \ref{TEO2.6}, we have
\[
|\psi_i(t)|_H \leq 2\dfrac{\|f\|_1}{\alpha} + \Gamma =  r_0, \quad \text{for all} \quad t \in [0, t_1] \quad \text{and} \quad i=1,2.
\]

Given $\omega = (\widetilde{\mathcal{B}}, \widetilde{\mathcal{F}}) \in \mathcal{M}$, there is a sequence $\{s_n\}_{n\in\mathbb{N}}\subset\mathbb{R}_+$ such that $\widetilde{\mathcal{F}}(t, u) = \displaystyle\lim_{n\rightarrow+\infty}\mathcal{F}_{s_n}(t, u)$ for all $(t, u) \in \mathbb{R}_+\times H$. By the definition of $f$, we have
\[
|f(t, \sigma(t, \omega), \psi_1) - f(t,\sigma(t, \omega), \psi_2)|_H = |\widetilde{\mathcal{F}}_t(0, \psi_1) - \widetilde{\mathcal{F}}_t(0, \psi_2)|_H =
\]
\[
= \lim_{n\rightarrow +\infty}|\mathcal{F}_{s_n}(t, \psi_1) - \mathcal{F}_{s_n}(t, \psi_2)|_H = \lim_{n\rightarrow +\infty}|P\phi(s_n + t, \psi_1) - P\phi(s_n + t, \psi_2)|_H \leq \|P\|C|\eta(t)|_H.
\]

 Then, using \eqref{JER1} and the above estimative, we have
\[
\dfrac{d}{dt}|\eta(t)|^2_H = -2\langle A\eta(t), \eta(t)\rangle + 2\langle B(\sigma(t, \omega))(\psi_2, \psi_2) - B(\sigma(t, \omega))(\psi_1, \psi_1), \eta(t)\rangle +
\]
\[
+ 2\langle f(t, \sigma(t, \omega), \psi_1) - f(t,\sigma(t, \omega), \psi_2), \eta(t)\rangle \leq
\]
\[
\leq -2\alpha|\eta(t)|^2_H + 4r_0\|B\|_{\infty}|\eta(t)|_H^2 + 2C\|P\||\eta(t)|^2_H
\]
\[
\leq -2\left(\alpha - 2\|B\|_{\infty}r_0 - C\right)|\eta(t)|^2_H = -2\beta |\eta(t)|^2_H,
\]
for all $t\in [0, t_1]$, where $\beta = \alpha - 2\|B\|_{\infty}r_0 - C$ and $\|P\|\leq 1$.
Hence, $|\eta(t)|^2_H \leq e^{-2\beta t}|\eta(0)|^2_H$, that is, 
\begin{equation}\label{EQ}
|\psi(t; 0, u_1, \omega) - \psi(t; 0, u_2, \omega)|_H \leq e^{-\beta t}|u_1-u_2|_H
\end{equation}
for all $t\in [0, t_1]$. Thus, if $0\leq t < t_1$, we get
\[
|\varphi(t, u_1, \omega) - \varphi(t, u_2, \omega)|_H =
|\psi(t; 0, u_1, \omega) - \psi(t; 0, u_2, \omega)|_H \leq e^{-\beta t}|u_1-u_2|_H.
\]

Let $t_1 \leq t < t_2$ and $\eta_1(t) = \varphi(t, u_1, \omega) - \varphi(t, u_2, \omega) = \psi(t; t_1, u_1^+, \omega) - \psi(t; t_1, u_2^+, \omega)$, where $\psi(t; t_1, u_i^+, \omega)$ is the solution of  \eqref{NSNA} such that $\psi(t_1) = u_i^+$ and $u_i^+ = \psi(t_1; 0, u_i, \omega) + I_1( \psi(t_1; 0, u_i, \omega))$, $i=1,2$.
 Following the steps above to show \eqref{EQ}, we obtain
\[
|\eta_1(t)|_H \leq e^{-\beta (t - t_1)}|\eta_1(t_1)|_H, \quad \text{for all} \; \;  t \geq t_1.
\]
On the other hand, we have
\begin{eqnarray*}
|\eta_1(t_1)|_H & = & |\psi(t_1; 0, u_1, \omega) + I_1(\psi(t_1; 0, u_1, \omega)) 
 - \psi(t_1; 0, u_2, \omega) - I_1(\psi(t_1; 0, u_2, \omega))|_H \\
& \leq & |\psi(t_1; 0, u_1, \omega) - \psi(t_1; 0, u_2, \omega)|_H+ |I_1(\psi(t_1; 0, u_1, \omega)) - I_1(\psi(t_1; 0, u_2, \omega))|_H\\
& \leq & (1 + C_2)|\psi(t_1; 0, u_1, \omega) - \psi(t_1; 0, u_2, \omega)|_H\\
 & \leq & (1 + C_2)e^{-\beta t_1}|u_1 - u_2|_H,
\end{eqnarray*}
where $C_2$ comes from Condition (H5).
Consequently,
\[
|\varphi(t, u_1, \omega) - \varphi(t, u_2, \omega)|_H \leq (1+C_2)e^{-\beta t}|u_1-u_2|_H, \quad \text{for all} \quad t_1 \leq t < t_2.
\]

Continuing with this process, if $t_k \leq t < t_{k+1}$ we get
\[
|\varphi(t, u_1, \omega) - \varphi(t, u_2, \omega)|_H  \leq (1+C_2)^ke^{-\beta t}|u_1-u_2|_H,
\]
for all $u_1, u_2\in \overline{B(0, \alpha^{-1}\|f\|_1)}$ and $\omega\in\mathcal{M}$.
\end{proof}

\vspace{.2cm}

\begin{remark} Theorems \ref{TeoNS}, \ref{TEO2.6}, \ref{TEO2.7} and \ref{TEO2.8} also hold for the non-impulsive system
\[
\left\{
\begin{array}{lll}
u' + Au + B(\sigma(t, \omega))(u,u) = f(t, \sigma(t, \omega), u), \ \ u \in H, \; t > 0, \vspace{1mm}\\
u(0) = u_0,
\end{array}
\right.
\]
with the obvious adaptations.
\end{remark}

\end{document}